\documentclass[11pt,a4paper]{article}
\usepackage{fullpage}
\usepackage{authblk}

\usepackage{epstopdf}
\usepackage{subfigure}
\usepackage{mathrsfs}
\usepackage{amsfonts}
\usepackage{amssymb}
\usepackage{eurosym}
\usepackage{bm}
\usepackage{multirow,bigstrut}
\usepackage{enumerate}
\usepackage{tikz}
\usepackage{graphicx}
\usepackage{stfloats}
\usepackage{array}
\usepackage{amsmath}
\usepackage[colorlinks,
            linkcolor=red,
            anchorcolor=blue,
            citecolor=green
            ]{hyperref}

\newenvironment{proof}{\noindent{\em \textbf{Proof.}}}{\quad \hfill$\Box$\vspace{2ex}}

\newtheorem{theorem}{Theorem}[section]
\newtheorem{definition}[theorem]{Definition}
\newtheorem{remark}[theorem]{Remark}
\newtheorem{corollary}[theorem]{Corollary}
\newtheorem{lemma}[theorem]{Lemma}

\numberwithin{equation}{section}

\def\i{{\bf i}}
\def\j{{\bf j}}
\def\k{{\bf k}}

\def\u{{\bm{\mu}}}
\def\a{{\bm{\alpha}}}
\def\b{{\bm{\beta}}}

\def\Sc{{\mathrm{Sc}}}
\def\Vec{{\mathrm{Vec}}}
\def \H {\mathbb{H}}

\def \X {\mathbf{X}}

\def \La {L^1(\mathbb{R}^2,\H)}
\def \Lb {L^2(\mathbb{R}^2,\H)}
\def\om{\omega}
\newcommand{\norm}[1]{\left\lVert#1\right\rVert}

\title{Plancherel theorem and quaternion Fourier transform for square integrable functions}
\author{Dong Cheng\thanks{chengdong720@163.com} }

\author{Kit Ian Kou\thanks{kikou@umac.mo}}

\affil{\normalsize{Department of Mathematics, Faculty of Science and Technology, University of Macau, Macao, China}}

\date{}

\begin{document}
	\maketitle
	\begin{abstract}
The quaternion Fourier transform (QFT), a generalization of the classical 2D Fourier transform, plays an increasingly active role in particular signal and color image processing. There tends to be an inordinate degree of interest placed on the properties of QFT. The classical convolution theorem and multiplication formula are only suitable for 2D Fourier transform of complex-valued signal, and do not hold for QFT of quaternion-valued signal. The purpose of this paper is to overcome these problems and  establish the Plancherel and inversion theorems of QFT in the square integrable signals space $L^2$. First, we investigate the behaviors of QFT in the integrable signals space $L^1$. Next, we deduce the energy preservation property which extends functions from $L^1$ to $L^2$ space. Moreover, some other important properties such as modified  multiplication   formula are  also analyzed for QFT.
\end{abstract}


\begin{msc}
	42A38, 42B10, 43A32, 43A50.
\end{msc}

\section{Introduction}
The quaternion Fourier transforms (QFTs) play a vital role in the representation of (hypercomplex) signals \cite{ell1992hypercomplex}.
They transform real (or quaternion-valued) 2D signals into quaternion-valued
frequency domain signals. The four components of the QFTs separate four
cases of symmetry of real signals instead of only two as in the complex Fourier Transform (FT) \cite{kou2014asymptotic}. In \cite{bulow1999hypercomplex, ell2007hypercomplex,assefa2010local} authors used the QFTs to extend color image analysis. The
study \cite{bas2003color,chen2014full} implemented the QFTs to design a color image digital watermarking
scheme. Researchers in \cite{bayro2007quaternion} applied the QFTs to image pre-processing and
neural computing techniques for speech recognition. Recently, certain asymptotic
properties of the QFTs were analyzed and a straightforward generalization
of the classical Bochner-Minlos theorem to the framework of quaternion
analysis was derived \cite{georgiev2013bochner, kou2014asymptotic}. The intermediate representations that combine the information of the signal and its QFT or quaternion linear canonical transform (QLCT) were discussed in \cite{bahri2010windowed,  fan2017quaternion} and the corresponding applications were also presented.

An excellent introduction to the history and developments of QFT was given by F. Brackx \emph{et al.} in \cite{brackx2013history}. Quaternions   were first applied to Fourier transform by Ernst \cite{ernst1987principles}
and Delsuc \cite{delsuc1988spectral} in the late 1980s. The early quaternion transforms were proposed for the application of nuclear magnetic resonance (NMR) imaging. Ell \cite{ell1992hypercomplex} in 1992 analyzed  two-sided QFT  and then   applied it to the color image processing. In \cite{ell2007hypercomplex}, Sangwine and Ell gave a detailed interpretation of the Fourier
coefficients obtained from a QFT. Following the pioneering works of   Ernst,  Ell and Sangwine, Hitzer  studied the  QFTs (including right-sided QFT and two-sided QFT) applied to quaternion-valued functions \cite{hitzer2007quaternion} and  presented  a series of further generalizations for  QFTs \cite{hitzer2013orthogonal,hitzer2016quaternion,Hitzer2016,hitzer2017general}. Subsequently, numerous fundamental results such as uncertainty principles and time-frequency distributions  for the classical Fourier transform were carried to the QFTs \cite{bahri2008uncertainty,YK2016,KYZ2016,kou2016uncertainty,fan2017quaternion}.

In classical Fourier theory, if $f$ is integrable ($L^1 (\mathbb{R}, \mathbb{C})$), then the Fourier transform $\widehat{f}(\xi)$ is well-defined by
\begin{equation} \label{1DFT}
\widehat{f}(\xi) =  \int_{\mathbb{R}}f(t) e^{-\i\xi t} dt.
\end{equation}
Moreover, if $\widehat{f}$ is also integrable, then $f(t)=\int_{\mathbb{R}}\widehat{f}(\xi) e^{\i\xi t} d\xi$ for almost every $t\in \mathbb{R}$. The integral (\ref{1DFT}) defining the Fourier transform is not suitable for square integrable functions ($L^2 (\mathbb{R}, \mathbb{C})$). But
there is a natural and elegant theory of Fourier transforms for square integrable functions. 
In fact, the  Fourier transform $\mathfrak{F}$ of $L^2 (\mathbb{R}, \mathbb{C})$ is not only isometric but also onto. This desirable property is very important in signal processing, especially in time-frequency analysis \cite{cohen1995time,grochenig2013foundations}.

One may naturally ask  what are the analogous results for QFTs of quaternion-valued signals? The previous contributions on inversion theorem and energy-preserved property of QFTs were developed in \cite{ell1992hypercomplex,ell1993quaternion, sangwine2012complex, bulow1999hypercomplex, hitzer2007quaternion}. On one hand, however, the existing results are not well established systematically. For instance, the definitions of QFTs  are  not well defined for square integrable functions.  The definitions of QFTs in $L^2$ ought to start from its dense subset $L^1\cap L^2$.  The previous studies didn't mention how to define the QFTs for square integrable
functions.  It should be noted that QFTs are different from the classical Fourier transform even in the one dimensional case.  In Theorem 4.2 of \cite{ell1992hypercomplex}, Parseval's identity of QFT  was established by using standard Parseval's identity of FT twice. Such a proof is inadequate.  On the other hand, the prerequisites of setting
up the established theorems were not studied systematically as well.  For instance, some authors proved Parseval's identity by using
\begin{equation}\label{Parseval id}
\begin{split}
\norm{f}_2^2&=  \int_{\mathbb{R}^2}f(x_1,x_2)\overline{f(x_1,x_2)} dx_1dx_2 \\
& = \int_{\mathbb{R}^2}\left(\int_{\mathbb{R}^2}(\mathcal{F}_rf)(\omega_1,\omega_2)e^{\j 2\pi \omega_2 x_2}e^{\i 2\pi \omega_1 x_1}d\omega_1d\omega_2\right )\overline{f(x_1,x_2)}dx_1dx_2\\
&=\int_{\mathbb{R}^2}(\mathcal{F}_rf)(\omega_1,\omega_2)\left(\overline{\int_{\mathbb{R}^2}f(x_1,x_2)e^{-\i 2\pi \omega_1 x_1} e^{-\j 2\pi \omega_2 x_2}dx_1dx_2}\right )d\omega_1d\omega_2\\
&=\int_{\mathbb{R}^2}(\mathcal{F}_rf)(\omega_1,\omega_2) \overline{(\mathcal{F}_rf)(\omega_1,\omega_2)}  d\omega_1d\omega_2=\norm{\mathcal{F}_rf }_2^2,
\end{split}
\end{equation}
where $f\in L^2(\mathbb{R}^2,\mathbb{H})$ and $\mathcal{F}_rf$ is the right-sided QFT (RQFT) of $f$ given by
\begin{equation*}\label{firstRQFTdef}
(\mathcal{F}_rf)(\om_1,\om_2):=\frac{1}{2\pi} \int_{\mathbb{R}^2}f(x_1,x_2)\displaystyle e^{-\i\om_1x_1} e^{-\j\om_2x_2}dx_1dx_2.
\end{equation*}

Even if the RQFT of $f\in L^2(\mathbb{R}^2,\mathbb{H})$  had been defined, the integral (\ref{Parseval id}) may not hold.

Therefore it is of great interest to progress the function theory of QFT for square integrable functions. To achieve this goal, we want to adopt the method of approximation to the identity by  ”good kernels”. This method   is   commonly used  in Fourier analysis \cite{stein1971introduction,rudin1987real}.
The contributions of this paper are summarized as follows.

\begin{itemize}
	\item[1.] In Fourier analysis, convolution theorem gives the connection between spatial and frequency domains of functions. It plays a vital role in verifying Plancherel theorem.  Unfortunately, the classical convolution theorem no longer holds for the QFTs.  In light of this, Theorem \ref{poisson-convolution} (Section \ref{S3-1}) is established to give the relations between the function $f$ (spatial domain) and its right-sided QFT (RQFT) (frequency domain).
	
	\item[2.] Another important property in Fourier analysis is the so-called {\it Multiplication Formula}. It is also not valid for the QFTs. Inspire of this, we derive the {\it Modified Multiplication Formula} for the RQFT in Theorem \ref{multiplication} (Section \ref{S3-1}). This formula is successfully applied to prove that the RQFT is {\it unitary} on the square integrable signals space $\Lb$ (Theorem \ref{Fourierunitary} in Section \ref{S3-2}).
	
	\item[3.] The inversions of RQFT have not been investigated for absolutely and square integrable quaterion-valued functions thoroughly in literature. We establish the inversion formulas (\ref{inversionl1}) and (\ref{inversionl2}) (Theorem \ref{inversionrightsided} in Section \ref{S3-1} and Theorem \ref{L2inverseSQFT} in Section \ref{S3-2}) for both function classes in the current paper.
	
	\item[4.] To deal with various types of QFTs, we exploit the relations between them (Theorem \ref{beta} in Section \ref{S4-1}).  Therefore, their inversion theorem on $\La$ along with Plancherel theorem on $\Lb$ are investigated, respectively.
	
\end{itemize}

The rest of the paper is organized as follows. Section \ref{S2} recalls some basic knowledge of quaternion algebra. Section \ref{S3} investigates the inversion theorem and Plancherel theorem of RQFT of $\Lb$. In Section \ref{S4}, we establish the relationships between RQFT and the  two-sided QFT (SQFT) and study the elementary properties of  SQFT.
Finally, discussions and conclusions are drawn in Section \ref{S6}.

\section{Preliminaries}\label{S2}

\subsection{Quaternion Algebra}\label{S2-1}

Throughout the paper, let
\begin{equation*}
\H :=\{q=q_0+\i q_1+\j q_2+\k q_3|~q_0,q_1,q_2,q_3\in\mathbb{R}\}
\end{equation*}
be the {\em{Hamiltonian skew field of quaternions}}, where the elements $\i$, $\j$ and $\k$ obey  Hamilton's multiplication rules:
\begin{equation*}
\i\j=-\j\i=\k,~~\j\k=-\k\j=\i,~~\k\i=-\i\k=\j,~~\i^2=\j^2=\i\j\k=-1.
\end{equation*}
For every quaternion $q=q_0+\underline{q}$, $\underline{q}=\i q_1+\j q_2+\k q_3$, the scalar and vector parts of $q$, are defined by $\Sc(q)=q_0$ and $\Vec(q)=\underline{q}$, respectively. If $q=\Vec(q)$, then $q_0=0$ and $q=\underline{q}$ is called a pure imaginary quaternion.
The quaternion conjugate is defined by $\overline{q}=q_0-\underline{q}=q_0-\i q_1-\j q_2-\k q_3$, and the norm $|q|$ of $q$ is defined by
$|q|^2={q\overline{q}}={\overline{q}q}=\sum_{m=0}^{m=3}{q_m^2}$.
Then we have
\begin{equation*}
\overline{\overline{q}}=q,~~~\overline{p+q}=\overline{p}+\overline{q},~~~\overline{pq}=\overline{q}~\overline{p},~~~|pq|=|p||q|,~~~~\forall p,q\in\H.
\end{equation*}
Using the conjugate and norm of $q$, one can define the inverse of $q\in\H\backslash\{0\}$ by $q^{-1}=\overline{q}/|q|^2$. The multiplication of two quaternions is noncommutative, but
\begin{equation}\label{cyclic multiplication symmetry}
\Sc(pq)=\Sc(qp)~~~~\forall p,q\in\H.
\end{equation}

The quaternion exponential function $e^{q}$ is defined by means of an infinite series as $$e^{q}:=\sum_{n=0}^\infty \frac{q^n}{n!}.$$
Analogous to the complex case one may derive a closed-form representation:
$$e^{q}=e^{q_0}(\cos|\underline{q}|+\frac{\underline{q}}{|\underline{q}|}\sin|\underline{q}|).$$

The unit sphere $S$ in $\H$ is defined by $S:=\{ q=q_0+\underline{q} \, : \, q =\Vec (q), |q|=1, q \in \H \}$. For a fixed $\u \in S$, let $\mathbb{C}_{\u}:=\{a+b\u:a,b\in\mathbb{R}, |\u|=1, \u=\Vec(\u) \}$, then it is easy to see that
$\mathbb{C}_{\u}$ is isomorphic to the complex plane.

\subsection{Quaternion module $L^p(\mathbb{R}^2,\H)$}\label{S2-2}

The left   quaternion module $L^p(\mathbb{R}^2,\H)(p=1,2)$  consists of all $\H$-valued  functions satisfying
\begin{equation*}
L^p(\mathbb{R}^2,\H):=\left\{f|f:\mathbb{R}^2\rightarrow \H,\|f\|_p^p:=\int_{\mathbb{R}^2}|f(x_1,x_2)|^{p}dx_1dx_2<\infty\right\}.
\end{equation*}
For $p=\infty$, like the complex case, the left quaternion module $L^\infty(\mathbb{R}^2,\H)$ consists of all essentially bounded measurable $\H$-valued functions on $\mathbb{R}^2$.
The left quaternionic inner product of $f,g\in L^2(\mathbb{R}^2,\H)$ is defined by
\begin{equation*}
<f,g>_{L^2(\mathbb{R}^2,\H)} :=\int_{\mathbb{R}^2}f(x_1,x_2)\overline{g(x_1,x_2)}dx_1dx_2.
\end{equation*}
This inner product leads to a scalar norm $\|f\|_2^2=<f,f>_{L^2(\mathbb{R}^2,\H)}$ which coincides with the usual $L^2$-norm for $f$, considered as a vector-valued function. In  fact, $L^2(\mathbb{R}^2,\H)$ is a left quaternionic Hilbert space with inner product $<\cdot,\cdot>_{L^2(\mathbb{R}^2,\H)}$ (see \cite{brackx1982clifford}).
For simplicity of notation, we use $<\cdot,\cdot>$  in place of $<\cdot,\cdot>_{L^2(\mathbb{R}^2,\H)}$ in the following.   It is not difficult to verify that
\begin{equation}\label{prop-of-inner-product}
<h f,qg>= h <f,g>\overline{q}, \qquad \mbox{for any } f, g \in L^2(\mathbb{R}^2,\H)  \mbox{ and }  h, q \in \H.
\end{equation}

\section{The right-sided quaternion Fourier transform} \label{S3}

The quaternion
Fourier transform (QFT) was first defined by Ell to  analyze linear time-invariant systems of partial differential equations \cite{ell1993quaternion}. Later, some constructive works related to the QFT and applications to signal and color image processing were studied \cite{sangwine1996fourier,bihan2003quaternion,bulow1999hypercomplex,hitzer2007quaternion,ell2007hypercomplex,bahri2008uncertainty,ell2013quaternion,kou2016envelope,YK2016,KYZ2016, YKZ2016,fan2017quaternion}). The non-commutativity of the quaternion multiplication leads to various types of quaternion Fourier transformations  (see \cite{ell2013quaternion}). In this
section, we  consider the right-sided QFT.

\subsection{The right-sided quaternion Fourier transform pairs in $\La$ }\label{S3-1}

Let us first review the definition of right-sided QFT (see for instance  \cite{ell2013quaternion,ell1993quaternion,georgiev2013bochner, KYZ2016}).
\begin{definition}[RQFT]
	For every $f\in\La$, the right-sided quaternion Fourier transform (RQFT) of $f$ is defined by
	\begin{equation}\label{RQFTdef}
	(\mathcal{F}_rf)(\om_1,\om_2):=\frac{1}{2\pi} \int_{\mathbb{R}^2}f(x_1,x_2)\displaystyle e^{-\i\om_1x_1} e^{-\j\om_2x_2}dx_1dx_2.
	\end{equation}
\end{definition}

It is well-known that the  convolution theorem  does not hold for the RQFT.  There is a pressing need for an alternative  tool to bridge the spatial and frequency domain of  functions.
Therefore we give the following result to connect their relationship.

\begin{theorem}\label{poisson-convolution}
	Let $p_\epsilon(x_1,x_2):=\frac{1}{\pi^2}\frac{\epsilon^2}{(\epsilon^2+x_1^2)(\epsilon^2+x_2^2)}$ be the Poisson kernel and  $P(\om_1,\om_2):= e^{-|\om_1|-|\om_2|}$.
	If $f\in\La$,  put $\widetilde{f}(x_1,x_2):=\overline{f(-x_1,-x_2)}$ and define $g(x_1,x_2):=(\widetilde{f}\ast f)(x_1,x_2)=\int_{\mathbb{R}^2}\overline{f(y_1,y_2)}f(x_1+y_1,x_2+y_2)dy_1dy_2$, then
	
	$$(f\ast p_\epsilon)(x_1,x_2)=\frac{1}{2\pi}\int_{\mathbb{R}^2} P(\epsilon \om_1,\epsilon \om_2)(\mathcal{F}_rf)(\om_1,\om_2)  e^{\j \om_2x_2} e^{\i \om_1x_1} d\om_1d\om_2.$$
	If $f$ is also contained in $\Lb$, then
	$$ \Sc((g\ast p_\epsilon )(0,0))= \int_{\mathbb{R}^2} P(\epsilon \om_1,\epsilon \om_2) \left |(\mathcal{F}_rf)(\om_1,\om_2) \right|^2  d\om_1d\om_2.$$
	
\end{theorem}

\begin{proof}
	We notice that
	\begin{equation*}
	\frac{\epsilon}{\epsilon^2+x^2}=\frac{1}{2}\int_{\mathbb{R}}e^{-|\epsilon \omega|}e^{\i \omega x}dx=\frac{1}{2}\int_{\mathbb{R}}e^{-|\epsilon \omega|}e^{\j \omega x}dx.
	\end{equation*}
	Hence we can write
	\begin{align*}
	p_\epsilon(x_1-y_1,x_2-y_2)
	&=  \frac{1}{2\pi^2}\cdot\frac{\epsilon}{\epsilon^2+(x_2-y_2)^2}\int_\mathbb{R} e^{-|\epsilon \om_1|} e^{-\i \om_1y_1} e^{\i \om_1x_1}d\om_1  \\
	&=  \frac{1}{2\pi^2} \int_\mathbb{R} e^{-|\epsilon \om_1|} e^{-\i \om_1y_1}\cdot\frac{\epsilon}{\epsilon^2+(x_2-y_2)^2} e^{\i \om_1x_1}d\om_1   \\
	&=  \frac{1}{4\pi^2}\int_\mathbb{R} e^{-|\epsilon \om_1|} e^{-\i \om_1y_1}\int_\mathbb{R} e^{-|\epsilon \om_2|} e^{\j \om_2(x_2-y_2)}d\om_2 e^{\i \om_1x_1}d\om_1 \\
	&= \frac{1}{4\pi^2}\int_{\mathbb{R}^2}P(\epsilon \om_1,\epsilon \om_2) e^{-\i \om_1y_1} e^{-\j\om_2y_2} e^{\j \om_2x_2} e^{\i \om_1x_1}d\om_1d\om_2.
	\end{align*}
	Using this  integral representation of Poisson kernel and taking the convolution of $f$ and $p_{\epsilon}$, we obtain
	\begin{align*}
	&~~~~ (f\ast p_\epsilon)(x_1,x_2) \\
	&= \int_{\mathbb{R}^2} dy_1dy_2f(y_1,y_2)\frac{1}{4\pi^2}\int_{\mathbb{R}^2}P(\epsilon \om_1,\epsilon \om_2) e^{-\i \om_1y_1} e^{-\j\om_2y_2} e^{\j \om_2x_2} e^{\i \om_1x_1}d\om_1d\om_2\\
	&=\frac{1}{4\pi^2}\int_{\mathbb{R}^2} P(\epsilon \om_1,\epsilon \om_2)d\om_1d\om_2\int_{\mathbb{R}^2}f(y_1,y_2)  e^{-\i \om_1y_1} e^{-\j \om_2y_2} e^{\j \om_2x_2} e^{\i \om_1x_1} dy_1dy_2\\
	&=\frac{1}{2\pi}\int_{\mathbb{R}^2}P(\epsilon \om_1,\epsilon \om_2)(\mathcal{F}_rf)(\om_1,\om_2)  e^{\j \om_2x_2} e^{\i \om_1x_1} d\om_1d\om_2.
	\end{align*}
	If $f\in \La \cap \Lb$, H{\"o}lder's inequality gives $g\in \La$. Note that
	\begin{align*}
	g(x_1,x_2)=(\widetilde{f}*f)(x_1,x_2)&=\int_{\mathbb{R}^2}\overline{f(-y_1,-y_2)}f(x_1-y_1,x_2-y_2)dy_1dy_2\\
	&= \int_{\mathbb{R}^2}\overline{f(y_1,y_2)}f(x_1+y_1,x_2+y_2)dy_1dy_2.
	\end{align*}
	Applying the  integral representation of Poisson kernel  we have
	\begin{align*}
	&~~~~\Sc((g*p_\epsilon)(0,0))\\
	&=\Sc\left(\int_{\mathbb{R}^2} g(y_1,y_2) p_\epsilon(-y_1,-y_2)dy_1dy_2\right)\\
	&= \Sc\left( \int_{\mathbb{R}^2}\left[\int_{\mathbb{R}^2}\overline{f(s_1,s_2)}f(s_1+y_1,s_2+y_2)ds_1ds_2\right] p_\epsilon(-y_1,-y_2)dy_1dy_2 \right) \\
	&= \Sc\left(  \int_{\mathbb{R}^2}ds_1ds_2\int_{\mathbb{R}^2}\overline{f(s_1,s_2)}f(s_1+y_1,s_2+y_2) p_\epsilon(-y_1,-y_2)dy_1dy_2 \right)  \\
	&=\Sc\left(\int_{\mathbb{R}^2}ds_1ds_2\int_{\mathbb{R}^2}\overline{f(s_1,s_2)}f(z_1,z_2) p_\epsilon(s_1-z_1,s_2-z_2)dz_1dz_2 \right)   \\
	&=  \Sc \left( \frac{1}{4\pi^2}\int_{\mathbb{R}^6}ds_1ds_2\overline{f(s_1,s_2)}f(z_1,z_2)dz_1dz_2P(\epsilon \om_1,\epsilon \om_2)     e^{-\i \om_1z_1} e^{-\j \om_2z_2} e^{\j \om_2s_2} e^{\i \om_1s_1} d\om_1d\om_2 \vphantom{\frac{1}{4\pi^2}}\right) \\
	&=\Sc\left(\frac{1}{4\pi^2}\int_{\mathbb{R}^6}d\om_1d\om_2P(\epsilon \om_1,\epsilon \om_2)ds_1ds_2 e^{\j \om_2s_2} e^{\i \om_1s_1}\overline{f(s_1,s_2)}f(z_1,z_2)
	e^{-\i \om_1z_1} e^{-\j \om_2z_2}dz_1dz_2\vphantom{\frac{1}{4\pi^2}} \right) \\
	&= \int_{\mathbb{R}^2} P(\epsilon \om_1,\epsilon \om_2) |(\mathcal{F}_rf)(\om_1,\om_2)|^2  d\om_1d\om_2
	\end{align*}
	which completes the proof.
\end{proof}

\begin{remark}
	The convolution operation  can be interpreted as a weighted average. After taking convolution operation with $p_{\epsilon}$, the new function $f\ast p_{\epsilon}$ is better than the original function $f$ because $f\ast p_{\epsilon}$ can be expressed by an inverse transform.
\end{remark}

To authors' knowledge, the inversions of RQFT have not been investigated for absolutely integrable quaternion-valued functions thoroughly. To proceed, we need the following lemmas.

\begin{lemma}\label{cauchy-sequence}
	Let $\Lambda$ be an index set. If $\{f_\lambda \}_{\lambda\in \Lambda}$ is a Cauchy sequence in $L^p(\mathbb{R}^2,\H)~ (p=1,2,\infty)$, with limit $f$, then $\{f_\lambda \}_{\lambda\in \Lambda}$  has a subsequence which converges pointwise for almost every $(x_1,x_2)\in \mathbb{R}^2$ to $f$.
\end{lemma}

\begin{lemma}\label{poisson-approximation}
	Let $p_\epsilon(x_1,x_2)=\frac{1}{\pi^2}\frac{\epsilon^2}{(\epsilon^2+x_1^2)(\epsilon^2+x_2^2)}$ be the Poisson kernel, then
	\begin{enumerate}[(\romannumeral1)]
		\item  $\lim_{\epsilon\rightarrow 0}\|f\ast p_{\epsilon}-f\|_p=0$ for every $f\in L^p(\mathbb{R}^2,\H)(p=1,2)$;
		\item  $\lim_{\epsilon\rightarrow 0}(f*p_\epsilon)(x_1,x_2)=f(x_1,x_2)$, if $f(x_1,x_2)\in L^\infty(\mathbb{R}^2,\H)$   is continuous at a point $(x_1,x_2)$.
	\end{enumerate}
\end{lemma}

The proofs of Lemmas \ref{cauchy-sequence} and \ref{poisson-approximation} in the 1D case are treated in  \cite{rudin1987real,stein1971introduction}. In fact, $f=f_0+\i f_1+\j f_2 +\k f_3\in L^p(\mathbb{R}^2,\H)(p=1,2,\infty)$ if and only if $f_m(m=0,1,2,3) \in L^p(\mathbb{R}^2,\mathbb{R})$. Then we can consider every $f\in L^p(\mathbb{R}^2,\H)(p=1,2,\infty)$ as a linear combination of real (complex) $L^p$ functions. More specifically, by the definition of norm of quaternions, we have $|f|^2= \sum_{m=0}^3|f_m |^2$ and therefore
\begin{equation*}
|f| \leq 2 \max\{|f_m|:m=0,1,2,3\}  \leq 2 \sum_{m=0}^3|f_m|.
\end{equation*}
Thus  $\norm{f}_{\infty}\leq 2 \sum_{m=0}^3\norm{f_m}_{\infty}$, $\norm{f}_1\leq 2 \sum_{m=0}^3\norm{f_m}_1$ and $\norm{f}_2^2\leq \sum_{m=0}^3\norm{f_m}_2^2$. That means that the $L^p$ ($p=1,2,\infty$) norm of quaternion-valued function is dominated  by the $L^p$ ($p=1,2,\infty$) norm of four real components. Hence in the multidimensional case it can be obtained from the similar results of 1D case.

\begin{remark}
	The function $p_\epsilon$ is  actually a family of good kernels.  Lemma \ref{poisson-approximation} indicates that $f*p_\epsilon$ converges  to $f$ as $\epsilon\to 0$ if $f$ satisfies   suitable conditions.
\end{remark}

By Lemmas \ref{cauchy-sequence}, \ref{poisson-approximation} and Theorem \ref{poisson-convolution}, we derive the inversion theorem of RQFT.

\begin{theorem}[Inversion of RQFT in $L^1$]\label{inversionrightsided}
	If $f$ and its RQFT $\mathcal{F}_rf\in L^1(\mathbb{R}^2,\H)$, moreover
	\begin{equation}\label{inversionl1}
	g(x_1,x_2)= \frac{1}{2\pi} \int_{\mathbb{R}^2}(\mathcal{F}_rf)(\om_1,\om_2) e^{\j \omega_2x_2} e^{\i \omega_1x_1}d\omega_1d\omega_2,
	\end{equation}
	then  $f(x_1,x_2)=g(x_1,x_2)$ for almost every $(x_1,x_2)\in \mathbb{R}^2$.
\end{theorem}

\begin{proof}
	From Theorem  \ref{poisson-convolution},
	\begin{equation}\label{poisson-convolution1}
	(f\ast p_\epsilon)(x_1,x_2)=\frac{1}{2\pi}\int_{\mathbb{R}^2} P(\epsilon \om_1,\epsilon \om_2)(\mathcal{F}_rf)(\om_1,\om_2)  e^{\j \om_2x_2} e^{\i \om_1x_1} d\om_1d\om_2.
	\end{equation}
	The integrands on the right side of (\ref{poisson-convolution1}) are bounded by $|(\mathcal{F}_rf)(\om_1,\om_2)|$, and since  $P(\epsilon \om_1,\epsilon \om_2)\rightarrow1$ as $\epsilon\rightarrow0$, then
	$$\frac{1}{2\pi}P(\epsilon \om_1,\epsilon \om_2)(\mathcal{F}_rf)(\om_1,\om_2)  e^{\j \om_2x_2} e^{\i \om_1x_1}\rightarrow \frac{1}{2\pi} (\mathcal{F}_rf)(\om_1,\om_2) e^{\j \om_2x_2} e^{\i \om_1x_1}$$
	as $\epsilon\rightarrow0$. Hence, the right side of (\ref{poisson-convolution1}) converges to $g(x_1,x_2)$, for every $(x_1,x_2)\in \mathbb{R}^2$,  by the dominated convergence theorem.
	
	From Lemmas \ref{cauchy-sequence} and \ref{poisson-approximation}, we see that that there is a sequence $\{\epsilon_n\}$ such that $\epsilon_n\rightarrow 0$ and
	$\lim_{n\rightarrow \infty}{(f*p_{\epsilon_n})(x_1,x_2)}={f(x_1,x_2)}$, for almost every $(x_1,x_2)\in \mathbb{R}^2$.
	Hence $f(x_1,x_2)=g(x_1,x_2)$ for almost every $(x_1,x_2)\in \mathbb{R}^2$.
\end{proof}

If $f \in L^1(\mathbb{R}^2,\H)$, Theorem \ref{inversionrightsided} indicates that $\mathcal{F}_r f =0$  for almost every $(\om_1,\om_2)\in \mathbb{R}^2$ implies $f=0$ for almost every $(x_1,x_2)\in \mathbb{R}^2$. Consequently, we have the following corollary.
\begin{corollary}[Uniqueness of RQFT]
	If $f,g\in L^1(\mathbb{R}^2,\H)$ and $(\mathcal{F}_rf)(\om_1,\om_2)=(\mathcal{F}_rg)(\om_1,\om_2)$ for almost every $(\om_1,\om_2)\in \mathbb{R}^2$, then $f(x_1,x_2)=g(x_1,x_2)$ for almost every $(x_1,x_2)\in \mathbb{R}^2$.
\end{corollary}

Now we see that, under suitable conditions,  the original signal $f$ can be reconstructed from $\mathcal{F}_rf$ by the inverse right-sided quaternion Fourier   transform (IRQFT).

\begin{definition}[IRQFT]\label{inversionRQFT}
	For every $f \in\La$, the inverse right-sided quaternion Fourier transform of $f$ is defined by
	\begin{equation*}
	(\mathcal{F}_r^{-1} f)(x_1,x_2):=\frac{1}{2\pi} \int_{\mathbb{R}^2} f(\om_1,\om_2) e^{\j\om_2x_2} e^{\i\om_1x_1}d\om_1d\om_2.
	\end{equation*}
\end{definition}
\begin{remark}
	Both $\mathcal{F}_r$ and $\mathcal{F}_r^{-1}$ are bounded linear transformation from $\La$ into   $L^\infty(\mathbb{R}^2,\H)$. We will see later, both $\mathcal{F}_r|_{L^1\cap L^2}$ and
	$\mathcal{F}_r^{-1}|_{L^1\cap L^2}$ can be extended to $\Lb$. As an operator on  $\Lb$,  $\mathcal{F}_r^{-1}$ is the inversion of $\mathcal{F}_r$.
\end{remark}

Next we present some properties of IRQFT.
\begin{theorem}\label{derivative}
	Suppose  that $F_r\in\La$ and $f=\mathcal{F}_r^{-1} F_r$.
	\begin{enumerate}[(\romannumeral1)]
		\item If $G_{1r}(\om_1,\om_2)=F_r(\om_1,\om_2)\i \om_1\in\La$, then the partial derivative of $f$ with respect to $x_1$ exists and $\frac{\partial f}{\partial x_1} (x_1,-x_2)=(\mathcal{F}_r^{-1}G_{1r})(x_1,x_2)$.
		\item If $G_{2r}(\om_1,\om_2)=F_r(\om_1,\om_2)\j \om_2\in\La$, then the partial derivative of $f$ with respect to $x_2$ exists and $\frac{\partial f}{\partial x_2} (x_1,x_2)=(\mathcal{F}_r^{-1}G_{2r})(x_1,x_2)$.
		\item If $G_{3r}(\om_1,\om_2)=F_r(\om_1,\om_2)\k \om_1\om_2\in\La$, then $$-\frac{\partial^2 f }{\partial x_1\partial x_2} (x_1,-x_2)=(\mathcal{F}_r^{-1}G_{3r})(x_1,x_2).$$
	\end{enumerate}
\end{theorem}
\begin{proof}
	To prove (\romannumeral1), we note that for $s\neq x_1$,
	\begin{equation}\label{partial derivative}
	\frac{f(s,-x_2)-f(x_1,-x_2)}{s-x_1}=\frac{1}{2\pi}\int_{\mathbb{R}^2}F_r(\om_1,\om_2)e^{-\j\om_2x_2}\frac{e^{\i \om_1 s}-e^{\i\om_1 x_1}}{s-x_1}d\om_1\om_2.
	\end{equation}
	Since $\left|\frac{e^{\i \om_1 s}-e^{\i\om_1 x_1}}{s-x_1}\right |\leq |\om_1|$ for all $s\neq x_1$ and $\lim_{s\to x_1}\frac{e^{\i \om_1 s}-e^{\i\om_1 x_1}}{s-x_1}=\i \om_1 e^{\i\om_1 x_1}$, then the integrands on the right side of (\ref{partial derivative})   are bounded by   $G_{1r}(\om_1,\om_2)$. Let $s$ tends to $x_1$, by the  dominated convergence theorem, we conclude that
	\begin{equation*}
	\begin{split}
	\frac{\partial f}{\partial x_1} (x_1,-x_2) & =\frac{1}{2\pi}\int_{\mathbb{R}^2}F_r(\om_1,\om_2)e^{-\j\om_2x_2}\i \om_1 e^{\i\om_1 x_1}d\om_1\om_2 \\
	& =\frac{1}{2\pi}\int_{\mathbb{R}^2}F_r(\om_1,\om_2)\i \om_1 e^{\j\om_2x_2}e^{\i\om_1 x_1}d\om_1\om_2=(\mathcal{F}_r^{-1}G_{1r})(x_1,x_2).
	\end{split}
	\end{equation*}
	By similar arguments, we can prove (\romannumeral2) and (\romannumeral3). We omit the proofs here.
\end{proof}

An important result in Fourier analysis is the so-called {\it multiplication formula}. The generalization of multiplication formula  to the RQFT is developed in the following. To  proceed, we introduce the auxiliary transform of $f(x_1,x_2)=f_0(x_1,x_2)+\i f_1(x_1,x_2)+\j f_2(x_1,x_2)+\k f_3(x_1,x_2)$. It is defined by
\begin{equation}\label{at}
\a f(x_1,x_2):=f_0(x_1,x_2)+\i f_1(x_1,-x_2)+\j f_2(-x_1,x_2)+\k f_3(-x_1,-x_2).
\end{equation}
Then we obtain the following result.

\begin{theorem}[Modified Multiplication Formula] \label{multiplication}
	Suppose  that $f,g\in \La$, $ h:= \a g$ , $H_r:=\mathcal{F}_rh$ and $F_r:=\mathcal{F}_rf$,  then
	\begin{equation}\label{multiplication-formula}
	\int_{\mathbb{R}^2}F_r(x_1,x_2)g(x_1,x_2)dx_1dx_2= \int_{\mathbb{R}^2}f(x_1,x_2)H_r(x_1,x_2)dx_1dx_2.
	\end{equation}
	Moreover, if $g$ belongs to $\Lb$, then we have $\norm{g}_2=\norm{h}_2$.
\end{theorem}

\begin{proof}
	Write $g=g_0 +\i g_1 +\j g_2 +\k g_3 $, then
	\begin{align*}
	&~~~~ \int_{\mathbb{R}^2}  e^{-\i\om_1x_1} e^{-\j\om_2x_2} g(\om_1,\om_2)d\om_1d\om_2 \\
	&=\int_{\mathbb{R}^2}g_0(\om_1,\om_2)  e^{-\i\om_1x_1} e^{-\j\om_2x_2} d\om_1d\om_2 +  \int_{\mathbb{R}^2}\i g_1(\om_1,\om_2)  e^{-\i\om_1x_1} e^{\j\om_2x_2} d\om_1d\om_2\\
	&~~ +  \int_{\mathbb{R}^2}\j g_2(\om_1,\om_2) e^{\i\om_1x_1} e^{-\j\om_2x_2} d\om_1d\om_2+ \int_{\mathbb{R}^2}\k g_3(\om_1,\om_2) e^{\i\om_1x_1} e^{\j\om_2x_2} d\om_1d\om_2\\
	&= \int_{\mathbb{R}^2}g_0(\om_1,\om_2) e^{-\i\om_1x_1} e^{-\j\om_2x_2} d\om_1d\om_2 +  \int_{\mathbb{R}^2}\i g_1(\om_1,-\om_2) e^{-\i\om_1x_1} e^{-\j\om_2x_2} d\om_1d\om_2\\
	&~~ +  \int_{\mathbb{R}^2}\j g_2(-\om_1,\om_2) e^{-\i\om_1x_1} e^{-\j\om_2x_2} d\om_1d\om_2+ \int_{\mathbb{R}^2}\k g_3(-\om_1,-\om_2) e^{-\i\om_1x_1} e^{-\j\om_2x_2} d\om_1d\om_2\\
	&= \int_{\mathbb{R}^2} h(\om_1,\om_2) e^{-\i\om_1x_1} e^{-\j\om_2x_2} d\om_1d\om_2=H_r(x_1,x_2).
	\end{align*}
	Applying Fubini's theorem, we have
	\begin{align*}
	&~~~~ \int_{\mathbb{R}^2}F_r(\om_1,\om_2)g(\om_1,\om_2)d\om_1d\om_2 \\
	&=  \int_{\mathbb{R}^2}  \left ( \int_{\mathbb{R}^2}f(x_1,x_2) e^{-\i\om_1x_1} e^{-\j\om_2x_2}dx_1dx_2\right )g(\om_1,\om_2)d\om_1d\om_2 \\
	&=  \int_{\mathbb{R}^2} f(x_1,x_2) \left ( \int_{\mathbb{R}^2} e^{-\i\om_1x_1} e^{-\j\om_2x_2}g(\om_1,\om_2)d\om_1d\om_2\right )dx_1dx_2 \\
	&= \int_{\mathbb{R}^2}f(x_1,x_2)H_r(x_1,x_2)dx_1dx_2.
	\end{align*}
	If $g\in \Lb$, it is easy to verify  that $\norm{g}_2=\norm{h}_2$ by definition of auxiliary transform (\ref{at}).
\end{proof}
\begin{remark}
	The  multiplication formula of complex Fourier transform has the form $$\int_{\mathbb{R}}\widehat{f}(x)g(x)dx=\int_{\mathbb{R}} f(\omega)\widehat{g}(\omega)d\omega,$$
	where $f,g\in\La$. This standard formula is not valid for RQFT of integrable functions. Using the  auxiliary
	transform $\a$, we obtain an   analogous formula ({\ref{multiplication-formula}}) for quaternion-valued integrable functions.  We name it the {\it modified multiplication formula of RQFT}.
\end{remark}

\subsection{The Plancherel Theorem Associated with RQFT}\label{S3-2}

In the complex case, the Plancherel theorem states that if $f\in L^1\cap L^2$, and it turns out that $\widehat{f}\in L^2$ and $\|\widehat{f}\|_2=\|f\|_2$,  where $\widehat{f}$ is the classical Fourier transform of $f$. Moreover, this  isometry of $L^1\cap L^2$  into
$L^2$ extends to an isometry of $L^2$  onto $L^2$ , and this extension defines the Fourier
transform  of every $f\in L^2$.  The convolution theorem plays a vital role in proving the Plancherel theorem (see  \cite{rudin1987real,stein1971introduction}). However, the classical convolution theorem no longer holds for the QFT.  The Plancherel theorem of QFT was discussed in recent research papers \cite{bulow1999hypercomplex},\cite{hitzer2007quaternion},\cite{bahri2008uncertainty}. We give a   restatement  of Plancherel theorem  here since the prerequisites for setting up of  the theorem may not be put forward so clearly  in recent research papers. It is probably worth pointing out that Theorem \ref{poisson-convolution} plays a key role in our proof.

\begin{theorem}\label{Parseval}
	If $f\in L^1(\mathbb{R}^2,\H)\cap L^2(\mathbb{R}^2,\H)$, then  $\mathcal{F}_rf\in  L^2(\mathbb{R}^2,\H)$ and
	Parseval's identity $\|\mathcal{F}_rf \|^2_2=\norm{f}^2_2 $ holds.
\end{theorem}

\begin{proof}
	We fix $f \in L^1(\mathbb{R}^2,\H)\cap L^2(\mathbb{R}^2,\H)$, put $\widetilde{f}(x_1,x_2):=\overline{f(-x_1,-x_2)}$ and define
	\begin{equation*}
	g(x_1,x_2) =(\widetilde{f}*f)(x_1,x_2)=<\overline{f},\overline{f}_{-x_1,-x_2}> =\int_{\mathbb{R}^2}\overline{f(y_1,y_2)}f(x_1+y_1,x_2+y_2)dy_1dy_2,
	\end{equation*}
	where $\overline{f}_{-x_1,-x_2}(y_1,y_2)$ denotes a translation of $\overline{f}$.
	Since $(x_1,x_2)\mapsto \overline{f}_{-x_1,-x_2}$ is a continuous mapping of $\mathbb{R}^2$ into $L^2(\mathbb{R}^2,\H)$ and noticing that the continuity of the inner product, we see that $g(x_1,x_2)$ is a continuous function. By Cauchy-Schwarz inequality,
	\begin{equation*}
	|g(x_1,x_2)|\leq \|\overline{f}_{-x_1,-x_2}\|_2\|f\|_2=\|f\|^2_2.
	\end{equation*}
	Thus $g$ is bounded. Furthermore, $g\in L^1(\mathbb{R}^2,\H)$ since $f\in L^1(\mathbb{R}^2,\H)$ and $\widetilde{f}\in L^1(\mathbb{R}^2,\H)$.

	Since $g$ is continuous and bounded, Lemma \ref{poisson-approximation} shows that
	\begin{equation*}
	\lim_{\epsilon\rightarrow 0}\Sc((g*p_\epsilon)(0,0))=\Sc(g(0,0))=\|f\|^2_2.
	\end{equation*}
	On the other hand,
	since $g\in L^1(\mathbb{R}^2,\H)$, by Theorem \ref{poisson-convolution},
	\begin{equation*}
	\Sc\left((g*p_\epsilon)(0,0)\right) = \int_{\mathbb{R}^2} P(\epsilon \om_1,\epsilon \om_2) |(\mathcal{F}_rf)(\om_1,\om_2)|^2 d\om_1d\om_2.
	\end{equation*}
	Since $0\leq P(\epsilon \om_1,\epsilon \om_2)  |(\mathcal{F}_rf)(\om_1,\om_2)|^2$ increases to $ |(\mathcal{F}_rf)(\om_1,\om_2)|^2$ as $\epsilon\rightarrow0$, the monotone convergence theorem gives
	\begin{align*}
	\lim_{\epsilon\rightarrow 0}\Sc((g*p_\epsilon)(0,0)) &=   \lim_{\epsilon\rightarrow 0} \int_{\mathbb{R}^2} P(\epsilon \om_1,\epsilon \om_2) |(\mathcal{F}_rf)(\om_1,\om_2)|^2 d\om_1d\om_2 \\
	&= \int_{\mathbb{R}^2}\lim_{\epsilon\rightarrow 0}  P(\epsilon \om_1,\epsilon \om_2) |(\mathcal{F}_rf)(\om_1,\om_2)|^2 d\om_1d\om_2 \\
	&=   \int_{\mathbb{R}^2} |(\mathcal{F}_rf)(\om_1,\om_2)|^2 d\om_1d\om_2=\|\mathcal{F}_rf\|^2_2.
	\end{align*}
	Therefore, $\mathcal{F}_rf \in L^2(\mathbb{R}^2,\H)$ and $\|\mathcal{F}_rf\|^2_2=\norm{f}^2_2$.
\end{proof}

From Theorem \ref{Parseval}, $\mathcal{F}_r|_{L^1\cap L^2}$ is an isometry of $\La\cap \Lb$  into
$\Lb$. Since  $\La\cap \Lb$ is a dense subset of $\Lb$. Therefore, there exists a unique bounded (continuous) extension, $\Psi_r$, of   $\mathcal{F}_r|_{L^1\cap L^2}$ to all of $\Lb$.
If $f\in \Lb$,  $F_r=\Psi_r f$ is defined by the $L^2$ limit of  the sequence $\{ \mathcal{F}_rf_k\} $, where $\{f_k\}$ is any sequence in $\La\cap\Lb$ converging to $f$ in the $L^2$ norm. If we choose  $f_k(x_1,x_2)=f(x_1,x_2)\chi_{[-k,k]^2}(x_1,x_2)$, we have
$$F_r(\om_1,\om_2)=\mathop{\mathrm{l.i.m.}}\limits_{k\rightarrow\infty}\frac{1}{2\pi}\int_{[-k,k]^2}f(x_1,x_2) e^{-\i x_1\om_1} e^{-\j x_2\om_2}dx_1dx_2,$$
where $f=\mathop{\mathrm{l.i.m.}}\limits_{k\rightarrow\infty}f_k$ means $\|f-f_k\|_2\rightarrow0$ as $k\rightarrow\infty$.

We call $F_r=\Psi_r f$ is the RQFT of $f$ on $L^2(\mathbb{R}^2,\H)$.  The multiplication formula (\ref{multiplication-formula}) easily extends to $\Lb$.  The left $\H$-linear operator $\Psi_r$ on $\Lb$ is an isometry. Therefore, $\Psi_r$ is a one-to-one mapping.  Moreover, we can show that $\Psi_r$ is onto.

\begin{theorem}\label{Fourierunitary} [Unitary]
	The RQFT $\Psi_r$  is a unitary operator on $\Lb$.
\end{theorem}
\begin{proof}
	We first show that the range of $\Psi_r$ denoted by $R(\Psi_r)$ is a closed subspace of $\Lb$. Let $\{F_k\}$ be a sequence  in $R(\Psi_r)$  converging to $F_r$ in $L^2$ norm sense. We now prove that $F_r \in R(\Psi_r)$.  Suppose that $\Psi_r f_k=F_k$. Then the isometric property shows that  $\Psi_r$ is continuous and $ \{f_k\}$ is also a Cauchy sequence.  The completeness of $L^2(\mathbb{R}^2,\H)$ implies  that $\{f_k\}$ converges to some $f\in L^2(\mathbb{R}^2,\H)$, and the continuity of $\Psi_r$ shows that $\Psi_r f =\mathop{\mathrm{l.i.m.}}\limits_{k\rightarrow\infty}\Psi_r f_k=F_r$.
	
	If $R(\Psi_r)$ were not all of $\Lb$, as every  closed subspace of Hilbert space $\Lb$  has an orthogonal complement, we could find a function $u$ such that $$ \int_{\mathbb{R}^2}F_r(x_1,x_2)\overline{u(x_1,x_2)}dx_1dx_2=0$$
	for all $f\in \Lb$ and $\norm{u}_2\neq 0$. Let $g=\overline{u}, h=\a g$, by multiplication formula (\ref{multiplication-formula}),
	\begin{equation*}
	\int_{\mathbb{R}^2}f(x_1,x_2)H_r(x_1,x_2)dx_1dx_2=\int_{\mathbb{R}^2}F_r(x_1,x_2)g(x_1,x_2)dx_1dx_2= 0
	\end{equation*}
	for all $f\in \Lb$. Pick $f=\overline{H_r}$,  this implies that $H_r(x_1,x_2)=0$  for almost every $(x_1,x_2)\in \mathbb{R}^2$, contradicting the  fact that $\norm{H_r}_2=\norm{h}_2=\norm{g}_2=\norm{u}_2\neq 0$.
\end{proof}

The next result shows that the mapping $\Psi_r$ is a Hilbert space isomorphism   of $L^2(\mathbb{R}^2,\H)$, that is, preserving inner product.

\begin{theorem}\label{keepinnerproduct}
	Let $f,g\in L^2(\mathbb{R}^2,\H)$ and $F_r :=\Psi_r f, G_r :=\Psi_r g$. Then
	\begin{equation*}
	\int_{\mathbb{R}^2}f(x_1,x_2)\overline{g(x_1,x_2)}dx_1dx_2=\int_{\mathbb{R}^2}F_r(\om_1,\om_2)\overline{G_r(\om_1,\om_2)}d\om_1d\om_2.
	\end{equation*}
\end{theorem}

\begin{proof}
	Let
	\begin{equation*}
	p_0+\i p_1+\j p_2+\k p_3=\int_{\mathbb{R}^2}f(x_1,x_2)\overline{g(x_1,x_2)}dx_1dx_2
	\end{equation*}
	and
	\begin{equation*}
	q_0+\i q_1+\j q_2+\k q_3=\int_{\mathbb{R}^2}F_r(\om_1,\om_2)\overline{G_r(\om_1,\om_2)}d\om_1d\om_2.
	\end{equation*}
	From the Parseval's identity, we  have
	\begin{eqnarray*}
		\|f+g\|_2^2 &=& \|f\|_2^2+\|g\|_2^2+2p_0 \\
		&=&  \|F_r+G_r\|_2^2 \\
		&=& \|F_r\|_2^2+\|G_r\|_2^2+2q_0 .
	\end{eqnarray*}
	Thus $p_0=q_0$.
	By using (\ref{prop-of-inner-product}) and applying Parseval's identity to $\|f+\i g\|_2^2=\|F_r+\i G_r\|_2^2$, $\|f+\j g\|_2^2=\|F_r+\j G_r\|_2^2$, $\|f+\k g\|_2^2=\|F_r+\k G_r\|_2^2$ respectively, we can get
	\begin{equation*}
	p_m=q_m,~~~(m=1,2,3)
	\end{equation*}
	which completes the proof.
\end{proof}

We have shown that  $\Psi_r$  is  unitary  on $\Lb$. Thus $\Psi_r^{-1}F$ is uniquely determined for every $F\in\Lb$. The following result gives the explicit expressions for $\Psi_r^{-1}F$.

\begin{theorem}\label{L2inverseSQFT}[Inversion of RQFT in $L^2$]
	The inverse $f=\Psi_r^{-1}F_r$ is the $L^2$ limit of the sequence $\{\mathcal{F}_r^{-1}F_{rk}\}_k$, where $\{F_{rk}\}_k$ is any sequence in $\La\cap\Lb$ converging to $F_r$ in the $L^2$ norm. If we choose $F_{rk}=F_r\chi_{[-k,k]^2}$, then we have
	\begin{eqnarray}\label{inversionl2}
	f(x_1,x_2)=\mathop{\mathrm{l.i.m.}}\limits_{k\rightarrow\infty}\frac{1}{2\pi}\int_{[-k,k]^2}F_r(\om_1,\om_2) e^{\j x_2\om_2} e^{\i x_1\om_1}d\om_1d\om_2.\end{eqnarray}
	In particular, if $F_r\in L^1(\mathbb{R}^2,\H)\cap L^2(\mathbb{R}^2,\H)$, then $$f(x_1,x_2)=\frac{1}{2\pi}\int_{\mathbb{R}^2}F_r(\om_1,\om_2) e^{\j x_2\om_2} e^{\i x_1\om_1}d\om_1d\om_2.$$
\end{theorem}
\begin{proof}
	For  $\Psi_r\in\mathcal{B}(\Lb)$, the quaternionic Riesz representation theorem (see \cite{brackx1982clifford})   guarantees that there exists a unique operator $\Psi_r^*\in\mathcal{B}(\Lb)$, which is called the adjoint of $\Psi_r$, such that for all $f,g\in \Lb$, $<\Psi_r f,g>=<f,\Psi_r^*g>$.  Since $\Psi_r$ is unitary, then $\Psi_r^{-1}=\Psi_r^*$. For any fixed $F_r\in\Lb$, let $\{F_{rk}\}_k$ be an arbitrary sequence  in $\La\cap\Lb$ converging to $F_r$ in the $L^2$ norm, then we have
	\begin{align*}
	<g,\Psi_r^*F_r> &=  <G_r,F_r> \\
	&=  \lim_{k\rightarrow \infty} <G_r,F_{rk}>\\
	&=  \lim_{k\rightarrow \infty} \int_{\mathbb{R}^2}\left(\int_{\mathbb{R}^2}g(x_1,x_2) e^{-\i x_1\om_1} e^{-\j x_2\om_2}dx_1dx_2\right)\overline{F_{rk}(\om_1,\om_2)}d\om_1d\om_2\\
	&= \lim_{k\rightarrow \infty} \int_{\mathbb{R}^2}g(x_1,x_2){\left(\overline{\int_{\mathbb{R}^2}{F_{rk}(\om_1,\om_2) e^{\j x_2\om_2} e^{\i x_1\om_1}}d\om_1d\om_2}\right)}dx_1dx_2\\
	&= \lim_{k\rightarrow \infty}<g,\mathcal{F}_r^{-1}F_{rk}>=<g,\mathop{\mathrm{l.i.m.}}\limits_{k\rightarrow\infty}\mathcal{F}_r^{-1}F_{rk}>
	\end{align*}
	for all $g\in\La\cap\Lb$. The last equality is a consequence of  the continuity of inner product. Thus $f=\Psi_r^{-1}F=\Psi_r^*F_r=\mathop{\mathrm{l.i.m.}}\limits_{k\rightarrow\infty}\mathcal{F}_r^{-1}F_{rk}$.
	In particular, if $F_r\in L^1(\mathbb{R}^2,\H)\cap L^2(\mathbb{R}^2,\H)$, then $$\displaystyle f(x_1,x_2)=\frac{1}{2\pi}\int_{\mathbb{R}^2}F_r(\om_1,\om_2) e^{\j x_2\om_2} e^{\i x_1\om_1}d\om_1d\om_2 $$
	which completes the proof.
\end{proof}
\begin{remark}
	Since $\Psi_r(\Psi_r^{-1})$ coincides with  $ \mathcal{F}_r(\mathcal{F}_r^{-1})$ in $\La\cap \Lb$. For simplicity of notations, in the following, by capital letter $F_r$, we mean the RQFT of $f\in \La \cup \Lb$ if no otherwise specified.
\end{remark}

\section{The SQFT}\label{S4}

In this section, we  study the two-sided (sandwich) quaternion Fourier transform (SQFT). Lets first review the definition of SQFT \cite{kou2016uncertainty, kou2016envelope}.

\subsection{The SQFT  Pairs in $\La$ }\label{S4-1}
\begin{definition}[SQFT]\label{SQFT}
	For every $f\in\La$, two-sided quaternion Fourier transform of $f$ is defined by
	\begin{equation*}
	(\mathcal{F}_sf)(\om_1,\om_2):=\frac{1}{2\pi} \int_{\mathbb{R}^2} e^{-\i\om_1x_1}f(x_1,x_2) e^{-\j\om_2x_2}dx_1dx_2.
	\end{equation*}
\end{definition}

Unlike the RQFT, SQFT is not a left $\H$-linear operator. But SQFT is left $\mathbb{C}_{\i}$-linear and right $\mathbb{C}_{\j}$-linear (Section \ref{S2-1}, $\mathbb{C}_{\u}$ with $\u={\i}$ or ${\j}$, respectively). Moreover,  the SQFT is related to the RQFT through the following transform.

\begin{definition}\label{betaoperator}
	If $f(x_1,x_2)= f_0(x_1,x_2)+\i f_1(x_1,x_2)+\j f_2(x_1,x_2)+\k f_3(x_1,x_2)\in L^p(\X,\H)(p=1,2)$, then the operator  $\b$ is  defined   by
	\begin{equation*}
	\b f(x_1,x_2):=f_0(x_1,x_2)+\i f_1(x_1,x_2)+\j f_2(-x_1,x_2)+\k f_3(-x_1,x_2).
	\end{equation*}
\end{definition}

\begin{theorem}\label{beta}
	Suppose  that $f,g\in L^p(\X,\H)(p=1,2)$. Then the following assertions hold.
	\begin{enumerate}[(\romannumeral1)]
		\item  The operator $\b$ is a left $\mathbb{C}_{\i}$-linear   bijection mapping on $L^p(\X,\H) $. Therefore the inverse of $\b$ can be well-defined, denoted it by $\b^{-1}$. Moreover, $\b^{-1} =\b$.
		\item If $f\in \La$ then
		\begin{equation}\label{relation-of-Fs-and-Fr}
		\mathcal{F}_sf=\mathcal{F}_r (\b f).
		\end{equation}
		Moreover, if $f$ is  $\mathbb{C}_{\i}$-valued or $f$ is even with respect to the first variable, then $\mathcal{F}_sf=\mathcal{F}_r  f$.
		\item If $f,g\in \Lb$ then $\Sc \left(<\b f, \b g>\right)=\Sc \left(<f, g>\right)$, $\Sc \left(\i <\b f, \b g>\right)=\Sc \left( \i <f, g>\right)$. In particular $\norm{f}_2=\norm{\b f}_2$.
	\end{enumerate}
\end{theorem}

\begin{proof}
	The assertion (\romannumeral1) is a direct consequence of the Definition \ref{betaoperator}. To prove assertion (\romannumeral2),  write $f$ in form of $  f_0+\i f_1+\j f_2+\k f_3$ and let $h:=\b f$. Then
	\begin{align*}
	&~~~~ \mathcal{F}_sf(\om_1,\om_2) \\
	&=  \int_{\mathbb{R}^2}  e^{-\i\om_1x_1} f(x_1,x_2)  e^{-\j\om_2x_2} dx_1dx_2 \\
	&= \int_{\mathbb{R}^2}f_0(x_1,x_2)  e^{-\i\om_1x_1} e^{-\j\om_2x_2} dx_1dx_2 +  \int_{\mathbb{R}^2}\i f_1(x_1,x_2)  e^{-\i\om_1x_1} e^{-\j\om_2x_2} dx_1dx_2\\
	&~~ +  \int_{\mathbb{R}^2}\j f_2(x_1,x_2)  e^{\i\om_1x_1} e^{-\j\om_2x_2} dx_1dx_2+ \int_{\mathbb{R}^2}\k f_3(x_1,x_2)  e^{\i\om_1x_1} e^{-\j\om_2x_2} ddx_1dx_2\\
	&= \int_{\mathbb{R}^2}f_0(x_1,x_2)  e^{-\i\om_1x_1} e^{-\j\om_2x_2} dx_1dx_2 +  \int_{\mathbb{R}^2}\i f_1(x_1,x_2)  e^{-\i\om_1x_1} e^{-\j\om_2x_2} dx_1dx_2\\
	&~~ +  \int_{\mathbb{R}^2}\j f_2(-x_1,x_2)  e^{-\i\om_1x_1} e^{-\j\om_2x_2} dx_1dx_2+ \int_{\mathbb{R}^2}\k f_3(-x_1,x_2)  e^{-\i\om_1x_1} e^{-\j\om_2x_2} ddx_1dx_2\\
	&= \int_{\mathbb{R}^2} h(x_1,x_2) e^{-\i\om_1x_1} e^{-\j\om_2x_2} dx_1dx_2= \mathcal{F}_rh(\om_1,\om_2) .
	\end{align*}
	It follows that $\mathcal{F}_sf=\mathcal{F}_r (\b f)$. If $f$ is  $\mathbb{C}_{\i}$-valued or $f$ is even with respect to the first variable, then $\b f= f$.  Thus $\mathcal{F}_sf=\mathcal{F}_r  f$.
	
	To prove the assertion (\romannumeral3),  let  $f,g \in \Lb$, then
	\begin{align*}
	&~~~~  \Sc \left(<\b f, \b g>\right)\\
	&=   \int_{\mathbb{R}^2}f_0(x_1,x_2)g_0(x_1,x_2)dx_1dx_2+ \int_{\mathbb{R}^2}f_1(x_1,x_2)g_1(x_1,x_2)dx_1dx_2 \\
	&~~  + \int_{\mathbb{R}^2}f_2(-x_1,x_2)g_2(-x_1,x_2)dx_1dx_2 +\int_{\mathbb{R}^2}f_3(-x_1,x_2)g_3(-x_1,x_2)dx_1dx_2\\
	&=  \int_{\mathbb{R}^2}f_0(x_1,x_2)g_0(x_1,x_2)dx_1dx_2+ \int_{\mathbb{R}^2}f_1(x_1,x_2)g_1(x_1,x_2)dx_1dx_2 \\
	&~~  + \int_{\mathbb{R}^2}f_2(x_1,x_2)g_2(x_1,x_2)dx_1dx_2 +\int_{\mathbb{R}^2}f_3(x_1,x_2)g_3(x_1,x_2)dx_1dx_2\\
	&=   \Sc \left(<f, g>\right).
	\end{align*}
	Since $\b$ is    left $\mathbb{C}_{\i}$-linear, then
	\begin{eqnarray*}
		\Sc \left( \i <f, g>\right)  &=& \Sc \left( <\i  f, g>\right)\\
		&=& \Sc \left( <\b(\i  f), g>\right) \\
		&=& \Sc \left( <\i  \b f, g>\right)=\Sc \left(\i <\b f, \b g>\right) .
	\end{eqnarray*}
	At last
	\begin{equation*}
	\norm{\b f}_2^2 =\norm{f_0}_2^2+\norm{f_1}_2^2+\norm{f_2}_2^2+\norm{f_3}_2^2=\norm{f}_2^2
	\end{equation*}
	which completes the proof.
\end{proof}
\begin{theorem}[Inversion of SQFT]\label{inversiontwosided}
	If $f, \mathcal{F}_sf\in L^1(\mathbb{R}^2,\H)$ and
	\begin{equation}\label{inversiontwosidedformula}
	g(x_1,x_2)= \frac{1}{2\pi} \int_{\mathbb{R}^2} e^{\i \omega_1x_1}(\mathcal{F}_sf)(\om_1,\om_2) e^{\j \omega_2x_2}d\omega_1d\omega_2,
	\end{equation}
	then   $f(x_1,x_2) =g(x_1,x_2)$ for almost every $(x_1,x_2)\in \mathbb{R}^2$. 
\end{theorem}

\begin{proof}
	By invoking assertion  (\romannumeral2) of Theorem \ref{beta}  we have $\mathcal{F}_s f=\mathcal{F}_r (\b f)$. Let $h:=\mathcal{F}_r^{-1}(\mathcal{F}_s f)$, then $(\b f)(x_1,x_2)=h(x_1,x_2)$ for almost every $(x_1,x_2)\in \mathbb{R}^2$  by Theorem \ref{inversionrightsided}. To prove $f(x_1,x_2) =g(x_1,x_2)$ for almost every $(x_1,x_2)\in \mathbb{R}^2$, it suffices to verify  that $\b g= h$. Note that  $$h(x_1,x_2)=\frac{1}{2\pi} \int_{\mathbb{R}^2}(\mathcal{F}_sf)(\om_1,\om_2) e^{\j\om_2x_2} e^{\i\om_1x_1}d\om_1d\om_2.$$
	It is easy to see that $\Sc (h(x_1,x_2))=\Sc(g(x_1,x_2))$ by invoking Eq. (\ref{cyclic multiplication symmetry}). Similarly we have $\Sc (h(x_1,x_2)\i)=\Sc(g(x_1,x_2)\i)$. Since
	\begin{equation*}
	h(x_1,x_2)\j =\frac{1}{2\pi} \int_{\mathbb{R}^2}(\mathcal{F}_sf)(\om_1,\om_2) e^{\j\om_2x_2}\j  e^{-\i\om_1x_1}d\om_1d\om_2.
	\end{equation*}
	Then
	\begin{align*}
	\Sc \left( h(x_1,x_2)\j \right) &=   \Sc \left( \frac{1}{2\pi} \int_{\mathbb{R}^2}(\mathcal{F}_sf)(\om_1,\om_2) e^{\j\om_2x_2}\j e^{-\i\om_1x_1}d\om_1d\om_2  \right) \\
	&=   \Sc \left( \frac{1}{2\pi} \int_{\mathbb{R}^2} e^{-\i\om_1x_1}(\mathcal{F}_sf)(\om_1,\om_2) e^{\j\om_2x_2}\j d\om_1d\om_2  \right) \\
	&=    \Sc \left(  g(-x_1,x_2)\j  \right).
	\end{align*}
	Analogously, we have $\Sc \left( h(x_1,x_2)\k \right) =\Sc \left(  g(-x_1,x_2)\k  \right)$. Then we conclude that $\b g= h$, which completes the proof.
\end{proof}
Therefore we can define the inverse  two-sided quaternion Fourier transform by Eq. (\ref{inversiontwosidedformula}) or equivalently  by  $\b ^{-1}\mathcal{F}_r^{-1}F_s$.
\begin{definition}[ISQFT]\label{inversionSQFT}
	For every $F_s\in\La$, the inverse two-sided quaternion Fourier transform of $F_s$ is defined by
	\begin{equation*}
	(\mathcal{F}_s^{-1}F_s)(x_1,x_2):=\frac{1}{2\pi} \int_{\mathbb{R}^2} e^{\i\om_1x_1}F_s(\om_1,\om_2) e^{\j\om_2x_2}d\om_1d\om_2.
	\end{equation*}
\end{definition}

\subsection{The Plancherel Theorem Associated with SQFT}\label{S4-2}

In Section \ref{S3-2}, we extend $\mathcal{F}_r|_{L^1\cap L^2}$ to $\Lb$. The RQFT on $\Lb$ has more symmetry than RQFT in $\La$. The relation  $\mathcal{F}_sf=\mathcal{F}_r (\b f)$ drives us to  extend $\mathcal{F}_s|_{L^1\cap L^2}$ to  $\Lb$.

\begin{definition}\label{L2SQFTdef}
	For every $f\in \Lb$, the  SQFT of $f$  is  defined  by
	\begin{equation}\label{L2SQFT}
	\Psi_s f := \Psi_r (\b f)
	\end{equation}
\end{definition}

In fact, we can define  $\Psi_s$ by taking $L^2$ norm limit of $f$ (Definition \ref{SQFT}). Eq. (\ref{L2SQFT}) gives a alternative expression   but actually equivalent form of $\Psi_s$.
\begin{theorem}\label{L2SQFT-prop}
	Suppose that $f, G_s\in \Lb$, Then the following assertions hold.
	\begin{enumerate}[(\romannumeral1)]
		\item The  $\Psi_s f$  defined by Eq. (\ref{L2SQFT}) is equal to the $L^2$ limit of  the sequence $\{ \mathcal{F}_sf_k\}_k $, where $\{f_k\}_k$ is any sequence in $\La\cap\Lb$ converging to $f$ in the $L^2$ norm. If, in addition, $f\in \La \cap \Lb$,  then $\Psi_s f=\mathcal{F}_sf$.
		\item  The SQFT $\Psi_s$ is a bijection on $\Lb$ and $\Psi_s^{-1} G_s =\b ^{-1}\Psi _r^{-1}G_s$. Furthermore, $\Psi_s^{-1} G_s$  is equal to the $L^2$ limit of the sequence $\{\mathcal{F}_s^{-1}G_{sk}\}_k$, where $\{G_{sk}\}_k$ is any sequence in $\La\cap\Lb$ converging to $G_s$ in the $L^2$ norm. If $G_s\in \La \cap \Lb$,  then $\Psi_s^{-1} G_s=\mathcal{F}_s^{-1}G_s $.
	\end{enumerate}
\end{theorem}
\begin{proof}
	The assertion (\romannumeral1) is a    consequence of (\ref{L2SQFT}) and definition of $\Psi_r$. The assertion (\romannumeral2) is a   consequence of (\ref{L2SQFT}) and Theorem \ref{L2inverseSQFT}.
\end{proof}

As an immediate consequence of Theorem \ref{keepinnerproduct} and Theorem \ref{L2inverseSQFT}, we present the following result.
\begin{theorem}
	If $f,g\in L^2(\mathbb{R}^2,\H)$, then $<\Psi_s f,g>=<\b f, \Psi_r^{-1} g>$.
\end{theorem}

Having defined the SQFT for functions in $\Lb$, we obtain the following Parseval's identity.
\begin{theorem}\label{Parseval Identity SQFT}
	Suppose that $f,g\in L^2(\mathbb{R}^2,\H)$, $F_s=\Psi_s f, G_s=\Psi_s g$. Let
	\begin{equation*}
	p_0+\i p_1+\j p_2+\k p_3=\int_{\mathbb{R}^2}f(x_1,x_2)\overline{g(x_1,x_2)}dx_1dx_2
	\end{equation*}
	and
	\begin{equation*}
	q_0+\i q_1+\j q_2+\k q_3=\int_{\mathbb{R}^2}F_s(\om_1,\om_2)\overline{G_s(\om_1,\om_2)}d\om_1d\om_2.
	\end{equation*}
	Then $\norm{F_s}_2=\norm{f}_2$ and  $p_m=q_m,~(m=0,1)$. Moreover, if both $f$ and $g$ are $\mathbb{C}_{\i}$-valued or even with respect to the first variable, then $p_m=q_m,~(m=0,1,2,3)$.
\end{theorem}
\begin{proof}
	Firstly,  we show that Parseval's identity of SQFT holds. Applying Parseval's identity of RQFT and (\romannumeral3) of Theorem \ref{beta}, we have
	\begin{align*}
	\norm{F_s}_2^2 &=  <\Psi_s f, \Psi_s f> \\
	&=  <\Psi_r (\b f), \Psi_r (\b f)> \\
	&= <\b f,\b f> \\
	&=  \norm{\b f}_2^2=\norm{f}^2_2.
	\end{align*}
	By invoking Parseval's identity  of SQFT to $\|f+ g\|_2^2=\|F_s+  G_s\|_2^2$, $\|f+\i g\|_2^2=\|F_s+\i G_s\|_2^2$ respectively, we  get $p_m=q_m,~(m=0,1)$.
	If  both $f$ and $g$ are $\mathbb{C}_{\i}$-valued or even with respect to the first variable, we have $\Psi_s f = \Psi_r f $ and $\Psi_s g = \Psi_r g $, therefore $<\Psi_s f, \Psi_s g >=<\Psi_r f, \Psi_r g >=<f, g>$, that is $p_m=q_m,~(m=0,1,2,3)$.
\end{proof}
\begin{remark}
	By applying Eq. (\ref{cyclic multiplication symmetry}),  Hitzer \cite{hitzer2007quaternion} proved  $ p_0=q_0$, it follows that  $\norm{F_s}_2=\norm{f}_2$.  We have shown that $p_1$ is also equal to $q_1$. One may be wondering what is the relationship between  $p_2$ and $q_2$ ($p_3$ and $q_3$)? At the moment, we do not know their relation. Since $\Psi_s$ is not left $\H$-linear, so $\|f+\j g\|_2^2\neq\|F_s+\j G_s\|_2^2$. In fact,  $\|f+g \j \|_2^2=\|F_s+ G_s \j\|_2^2$. It follows that $\Sc \left(<f,g\j>\right)=\Sc \left(<F_s,G_s \j>\right)$ rather than  $\Sc \left(<f, \j g>\right)=\Sc \left(<F_s,\j G_s >\right)$. However, $\Sc \left(<f, \j g>\right)=\Sc \left(<F_s,\j G_s >\right)$ is equivalent to $p_2=q_2$ by applying Eq. (\ref{prop-of-inner-product}).
\end{remark}

\section{Discussions and Conclusions}\label{S6}

Due to the non-commutativity  of multiplication of quaternions, there are different types of QFTs and we only consider two typical types of them. How about the rest of QFTs?

\begin{itemize}
	\item [1.]   The left-sided QFT (LQFT) $ e^{-\i\om_1x_1} e^{-\j\om_2x_2}f(\cdot)$ follows a similar pattern to RQFT, with kernel moved to left hand side. As left-sided QFT is right $\H$-linear, we only need to revise the definition of inner product in $\Lb$ to be  \begin{equation*}
	<f,g>_{L^2(\mathbb{R}^2,\H)}=\int_{\mathbb{R}^2}\overline{f(x_1,x_2)}g(x_1,x_2)dx_1dx_2.
	\end{equation*}
	Then  the results  of  RQFT still hold for LQFT case.
	\item [2.] If $\i, \j$ are substituted into $\u_1, \u_2$ respectively, where $\u_1, \u_2$ be any two orthogonal  unit pure imaginary quaternion, all of above results still hold. A    novel split called orthogonal 2D planes split (OPS) was introduced by Hitzer and Sangwine \cite{hitzer2013orthogonal}. Employing OPS, they further analyzed  a kind of new QFT form with respect to two arbitrary pure unit quaternions.   In \cite{sangwine2012complex}, the author proved the inversion theorem of two-sided discrete QFT, whereas  the imaginary units $\u_1, \u_2$  of their transform do not need to be orthogonal. In the discrete case, both of transform and inverse transform are  presented by finite summations. The interchange for the order of two finite summations  is permissible   and  finite sequences are always summable. So the situation of current paper is different from that in \cite{sangwine2012complex}. In the continuous case,  the proposed method of current paper is invalid for the  general imaginary units. We may consider this problem as a part of  our future work.
	\item [3.]   The types of $ e^{-\u\om_1x_1} e^{-\u\om_2x_2}f(\cdot), f(\cdot)  e^{-\u\om_1x_1} e^{-\u\om_2x_2}$ and $ e^{-\u\om_1x_1}f(\cdot) e^{-\u\om_2x_2}$ (single axis types)  obviously easier than the types in present paper (factored types). The single axis types of QFTs have similar properties to factor types of QFTs. Moreover,  the proofs will be simpler.
	\item [4.]  The inversion theorem and Plancherel theorem of transforms $ e^{-\u_1\om_1x_1-\u_2\om_2x_2}f(\cdot)$ and $ f(\cdot) e^{-\u_1\om_1x_1-\u_2\om_2x_2}$ (dual axis types) have not been worked out yet in this paper.
\end{itemize}

The linear canonical transform  (LCT), as a generalization of the classical Fourier transform, has more
degrees of freedom  than the FT and the FRFT, but with
similar computation cost as the conventional FT \cite{koc2008digital}. In \cite{kou2016uncertainty},  the authors generalized the classical LCT to the quaternionic
algebra and defined quaternion linear canonical transforms (QLCTs). The generalized analytic signal in 2{D} {QLCTs} domains was also applied to envelope detection in \cite{kou2016envelope}. Each type of QLCT   corresponds to a specific type of QFT. Based on the existing properties of QFTs, the properties of QLCTs can be established by building relationship between QLCTs and  QFTs. More detail about the properties of QLCTs for square integrable functions can be found in \cite{CK2016}.

Notable differences with the standard FT gives the QFT   an important role in scientific and engineering applications. Meanwhile  several novel tools (see e.g. Theorem 3.2, 3.12, 4.3) need to be proposed to prove the properties of QFT for square integrable functions. All of the newly obtained results (Theorem 3.2, Lemma 3.4 and 3.5, Theorem 3.7, 3.11, 3.14, 3.15, 3.16, 3.17, 4.8, 4.9) are original and different from the previous studies in  \cite{ell1992hypercomplex,ell1993quaternion, sangwine2012complex, bulow1999hypercomplex, hitzer2007quaternion} and they have enriched the  content of quaternion Fourier analysis.

Some previous results (see e.g. Theorem 4.1 and 4.2 in \cite{ell1992hypercomplex}) were established by simply using    results of the standard FT for complex-valued functions. However, it should be noted  that the partial QFT can not be simply viewed as the standard FT. The results of standard FT could be directly applied only if  both of the  function  and  the  kernel are complex-valued (i.e. belong to $\mathbb{C}_\u$). If only the kernel belongs to complex, the transform can only be viewed  as a single axis type QFT.  The single axis QFTs are different from the standard FT. For example, the standard convolution theorem and multiplication formula are invalid for the single axis QFTs. But these results are important to study the QFTs. Moreover, for Theorem 4.1  in \cite{ell1992hypercomplex}, the author didn't make any assumption about  the transformed function $H(\j \omega,\k v)$.  In fact, the assumption $ H\in L^1(\mathbb{R}^2,\mathbb{H}) $  is  essential.

It has been shown by several authors that the QFT can be split into the standard FTs (not directly view a QFT as two partial QFTs). We had thought to derive Plancherel theorem from existing results of the standard case. But there are two reasons why  we did not use this approach.

\begin{itemize}
	\item [1.]      When we consider inversion theorem on $ L^1(\mathbb{R}^2,\mathbb{H})$ and Plancherel theorem on $ L^2(\mathbb{R}^2,\mathbb{H})$, We need to know not only how to split the QFT $\mathcal{F}$ into complex FTs ($\mathfrak{F}_1,\mathfrak{F}_2, \cdots$), but also how to transform  $\mathcal{F}f$ back to  original $f$ by using  corresponding inverse transforms $\mathfrak{F}_1^{-1},\mathfrak{F}_2^{-1}, \cdots$. This goal is hard to achieve due to the non-commutativity.
	\item [2.]   Write $f=f_0+\i f_1+\j f_2+ \k f_3$ and
	$$e^{-\i 2\pi \omega_1 x_1} e^{-\j 2\pi \omega_2 x_2}=e^{-\i 2\pi (\omega_1 x_1+\omega_2x_2)}  \frac{1-\k}{2}+e^{-\i 2\pi  (\omega_1 x_1-\omega_2x_2)}  \frac{1+\k}{2}.$$
	Then   the right-sided QFT can be spilt into the standard FTs as follows.
	\begin{equation}\label{Split}
	\begin{split}
	&~~~~F_r(\omega_1,\omega_2)\\
	&=\int_{\mathbb{R}^2}f(x_1,x_2)e^{-\i 2\pi \omega_1 x_1} e^{-\j 2\pi \omega_2 x_2} dx_1dx_2\\
	&  =F_1(\omega_1,\omega_2) \frac{1-\k}{2}+\overline{F_2(\omega_1,\omega_2)} \frac{\j-\i}{2}+F_1(\omega_1,-\omega_2) \frac{1+\k}{2}+\overline{F_2(\omega_1,-\omega_2)} \frac{\j+\i}{2}
	\end{split}
	\end{equation}
	where $F_1$ and $F_2$ are standard FTs (both  of the   function  and the  kernel are complex-valued) and they are respectively  defined by
	\begin{equation*}
	F_1(\omega_1,\omega_2)=\int_{\mathbb{R}^2}(f_0(x_1,x_2)+\i f_1(x_1,x_2))e^{-\i 2\pi (\omega_1 x_1+\omega_2x_2)} dx_1dx_2
	\end{equation*}
	and
	\begin{equation*}
	F_2(\omega_1,\omega_2)=\int_{\mathbb{R}^2}(f_2(x_1,x_2)-\i f_3(x_1,x_2))e^{-\i 2\pi (\omega_1 x_1+\omega_2x_2)} dx_1dx_2.
	\end{equation*}
	By  Parseval's identity of the standard FT, $\norm{f}_2^2=\norm{f_0+\i f_1}_2^2+\norm{f_2-\i f_3}_2^2=\norm{F_1}_2^2+\norm{F_2}_2^2$. However, we can not conclude from Eq. (\ref{Split}) that
	$\norm{F_r}_2^2=\norm{F_1}_2^2+\norm{F_2}_2^2$.
\end{itemize}

In summary, we investigate  the behaviors of the QFTs (except for  dual axis types) on the space $\La$. They are reversible under the suitable condition  (the transformed function still integrable). The QFTs on $\Lb$ have symmetric property, for example, RQFT is unitary on  $\Lb$. We exploit the relationship between various types of QFTs. These relations are significant in the understanding of quaternion Fourier analysis.

\section{Acknowledgements}

The authors acknowledge financial support from the National Natural Science Funds 11401606, University of Macau MYRG2015-00058-L2-FST and the Macao Science and Technology Development Fund (FDCT/099/2012/A3 and FDCT/031/2016/A1).


\end{document}